\documentclass{article}

\usepackage[english]{babel}

\usepackage[a4paper,top=2cm,bottom=2cm,left=3cm,right=3cm,marginparwidth=1.75cm]{geometry}

\usepackage{amsmath}
\usepackage{amssymb}
\usepackage{amsthm}
\usepackage{mathtools}
\usepackage{graphicx}
\usepackage[colorlinks=true, allcolors=blue]{hyperref}
\usepackage{relsize}
\newtheorem{theorem}{Theorem}

\newtheorem{lemma}{Lemma}
\newtheorem{definition}{Definition}
\usepackage{amsfonts}
\usepackage{tikz}
\usetikzlibrary{fit,shapes.geometric}
\usepackage{float}
\usepackage{caption}
\usepackage{booktabs}
\usepackage{comment}
\usepackage{version}
\excludeversion{confidential}
\usepackage{hhline}
\includecomment{comment}
\usetikzlibrary{positioning}
\usepackage[numbers]{natbib}

\newcommand*\samethanks[1][\value{footnote}]{\footnotemark[#1]} 

\tikzset{every node/.style={circle}}
\tikzstyle{g0}=[scale=0.5,fill,outer sep=0]
\tikzstyle{g1}=[draw,inner sep=0,outer sep=0]

\title{Exact Counts of $C_{4}$s in Blow-Up Graphs}
\author{S.~Y. Chan\thanks{Deakin University, Geelong, Australia, School of Information Technology, Faculty of Science Engineering \& Built Environment, \underline{Australia}}        
    \and    
    K. Morgan\samethanks
    \and
    J.Ugon\samethanks[1]}

\date{}

\begin{document}
\maketitle

\begin{abstract}
Cycles have many interesting properties and are widely studied in many disciplines. In some areas, maximising the counts of $k$-cycles are of particular interest. A natural candidate for the construction method used to maximise the number of subgraphs $H$ in a graph $G$, is the \emph{blow-up} method. Take a graph $G$ on $n$ vertices and a pattern graph $H$ on $k$ vertices, such that $n\geq k$, the blow-up method involves an iterative process of replacing vertices in $G$ with a copy of the $k$-vertex graph $H$. In this paper, we apply the blow-up method on the family of cycles. We then present the exact counts of cycles of length 4 for using this blow-up method on cycles and generalised theta graphs. 
\end{abstract}

\section{Introduction}
In graph theory, the family of cycles are considered to have very rich structures. Cycles have many interesting properties and are of interest in many disciplines. Some of the many applications of cycles include periodic scheduling, the identification of weak interdependence in ecosystems and to determine chemical pathways in chemical networks \cite{adriaens19,kavitha09}.

In network analysis, cycles are used in modelling and studying communication systems, improve consensus network performance, to investigate faults and reliability and also study the topological features of such networks \cite{zelazo13}. In some cases, counting cycles is used as part of network analysis or message-passing algorithms \cite{karimi13,liu06,safar11}. Some other interesting problems arise in relation to counting $k$-cycles in graphs. For example, counting the number of 4-cycles in a tournament \cite{linial16} and enumerating cycles of a given length \cite{alon97}.

Further, some problems looks into maximising the number of cycles in graphs. There exists a kidney exchange program (KEP), which involves maximising the number of (directed) cycles to maximise the expected number of transplants \cite{alvelos16,biro09,pedrose14}.

In the areas of extremal graph theory, a construction method used to maximise counts of graphs is the \emph{blow-up} method. This method has been applied in graphs to study graph spectra \cite{oliveira14} and even the maximum induced density of graphs \cite{hatami2014}. This method was also applied as an approach to the Caccetta-H\"{a}ggkivst conjecture \cite{bondy97} and Johansson conjecture \cite{johansson00}.

Suppose we have graphs $G$ and $H$ on $n$ and $k$ vertices respectively, such that $n\geq k$. The blow-up method is an iterative process of replacing each vertex in $G$ with a \emph{copies} of $H$. If \emph{all} vertices in $G$ have been replaced with a copy of $H$, this is also known as a \emph{balanced} blow-up of $G$.

In this paper, we are interested in determining the exact number of induced $C_{4}$s in the nested blow-up graph. We will find the exact counts of $C_{4}$s in two different blow-up graphs, one in the nested blow-up of $C_{4}$s and the other in the theta graph $\Theta_{2,2,2}$. We give the associated formulas with respect to the level of blow-up $N$, which is defined in the later section.

\section{Notations and Definitions} 
In this section, we present some basic definitions and notations that are used in this paper. All graphs in this paper are simple. 

A graph $G=(V,E)$ consists of the (finite) vertex set $V$ and edge set $E$, which is a subset of all unordered pair of vertices. The \emph{order} of a graph is the number of vertices, whereas the \emph{size} of a graph is the number of edges. Let $u,v\in V(G)$, we say that $u$ is \emph{adjacent} to $v$ if there exists an edge $\{u,v\}\in E(G)$. We say that the edge $\{u, v\}$ is incident to vertices $u$ and $v$.

\begin{definition}[Graph composition]
Let $G$ and $H$ be simple graphs. The composition of graph $G$ and $H$, denoted $G[H]$, is the graph obtained by replacing each vertex $v\in V(G)$ by the graph $H_{v} \cong H$ and adding an edge between every vertex of $H_{u}$ and every vertex of $H_{v}$ where $\lbrace u, v\rbrace \in E(G).$ 
\end{definition}

\begin{definition}[Nested blow-up graph]
If $H$ is isomorphic to $G$ and each vertex $v\in V(G)$ is labelled $v\in \{ 0,1,2,\ldots,n-1 \}$, then $G_{0}=G$, $G_{1}=G[G]$,$\ldots$, $G_{N}=G[\underbrace{G[G[G[\ldots[G]]]]}_{N-1}]=G[G_{N-1}]$, where at each level $N$, the nested blow-up graph $G_{N}$ can be obtained from $G_{N-1}$ by replacing each vertex in $G$ by a copy of $G_{N-1}$ and replacing each edge by all possible edges between adjacent copies of $G_{N-1}$. We refer to each of such copies $G_{N-1}$ as \textbf{blobs} in our proofs in order to differentiate from other subgraphs in $G_{N}$ that may be isomorphic to $G_{N-1}$. 
\end{definition}

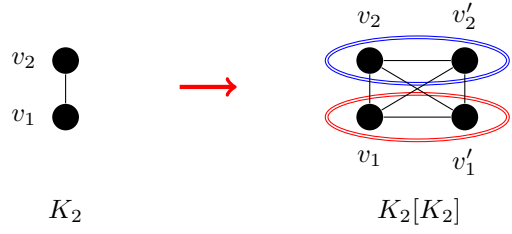
\begin{figure}
    \centering
    \begin{tikzpicture}
        [scale=.5,auto=left,every node/.style={circle,fill=black, minimum size=1em}]
        \node[label=left:$v_{1}$] (n0) at (0,0) {};
        \node[label=left:$v_{2}$] (n1) at (0,1.5) {};
        \node[fill=none] (nn) at (0,-2.5) {$K_{2}$};
        
        \draw (n0) edge (n1);
        
        \draw[ultra thick,red, ->] (3,0.8) -- (4.5,0.8);
        
        \begin{scope}[xshift=8cm]
        \node[label=below:$v_{1}$] (n0) at (0,0) {};
        \node[label=above:$v_{2}$] (n1) at (0,1.5) {};
        \node[label=below:$v_{1}'$] (n2) at (2.5,0) {};
        \node[label=above:$v_{2}'$] (n3) at (2.5,1.5) {};
        \node[fill=none] (nn1) at (1.3,-2.5) {$K_{2}[K_{2}]$};
        \node[draw=red,fill=none,double,fit=(n0) (n2) ,inner sep=1pt,ellipse] (tmp) {};
        \node[draw=blue,fill=none,double,fit=(n1) (n3) ,inner sep=1pt,ellipse] (tmp) {};
        \draw (n0) edge (n1);
        \draw (n1) edge (n2);
        \draw (n2) edge (n3);
        \draw (n0) edge (n3);
        \draw (n0) edge (n2);
        \draw (n1) edge (n3);
        \end{scope}
    \end{tikzpicture}
    \caption{An example of the nested blow-up of $K_{2}$. In $K_{2}[K_{2}]$ there are two blobs shown in blue and red, but 6 subgraphs that are isomorphic to $K_{2}$.  }
    \label{fig:my_label}
\end{figure}

\begin{figure}[H]
\centering
\begin{tikzpicture}[scale=0.8,auto=left,every node/.style={circle,fill=black}]
    \node (n0) at (0,0) {};
    \node (n1) at (1,-1) {};
    \node (n2) at (-1,-1) {};
    \node (n3) at (0,-2) {};
    
    \draw (n0) -- (n1) -- (n3) -- (n2) -- (n0);
    
    \draw[ultra thick,red, ->] (3,-1) -- (4,-1);

\begin{scope}[xshift=8cm, remember picture,
  inner/.style={scale=.5,circle,draw=black,fill=black,thick},
  outer/.style={scale=0.6,circle,draw=black,fill=none,thick}
  ]
  \node[outer,draw=black] (A) at (0,-2) {
    \begin{tikzpicture}
      \node [inner,draw=blue] (ai)  {};
      \node [inner,draw=blue,below=of ai] (aii) {};
      \node [inner,draw=blue,right=of aii] (aiii) {};
      \node [inner,draw=blue,right=of ai] (aiv) {};
      \draw[red, ultra thick] (ai) -- (aii) -- (aiii) -- (aiv) -- (ai);
    \end{tikzpicture}
  };
  \node[outer,draw=black] (B) at (-2,0) {
    \begin{tikzpicture}
      \node [inner,draw=blue] (bi)  {};
      \node [inner,draw=blue,below=of bi] (bii) {};
      \node [inner,draw=blue,right=of bii] (biii) {};
      \node [inner,draw=blue,right=of bi] (biv) {};
      \draw[red, ultra thick] (bi) -- (bii) -- (biii) -- (biv) -- (bi);
    \end{tikzpicture}
  };
    \node[outer,draw=black] (C) at (0,2) {
    \begin{tikzpicture}
        \node [inner,draw=blue] (ci)  {};
      \node [inner,draw=blue,below=of ci] (cii) {};
      \node [inner,draw=blue,right=of cii] (ciii) {};
      \node [inner,draw=blue,right=of ci] (civ) {};
      \draw[red, ultra thick] (ci) -- (cii) -- (ciii) -- (civ) -- (ci);
    \end{tikzpicture}
  };
     \node[outer,draw=black] (D) at (2,0) {
    \begin{tikzpicture}
        \node [inner,draw=blue] (di)  {};
      \node [inner,draw=blue,below=of di] (dii) {};
      \node [inner,draw=blue,right=of dii] (diii) {};
      \node [inner,draw=blue,right=of di] (div) {};
      \draw[red, ultra thick] (di) -- (dii) -- (diii) -- (div) -- (di);
    \end{tikzpicture}
    };
  \draw[line width=5mm] (A) -- (B) -- (C) -- (D) -- (A);
  \end{scope}
\end{tikzpicture}
\caption{Example of a nested blow-up graph of $C_{4}$.}
\end{figure}
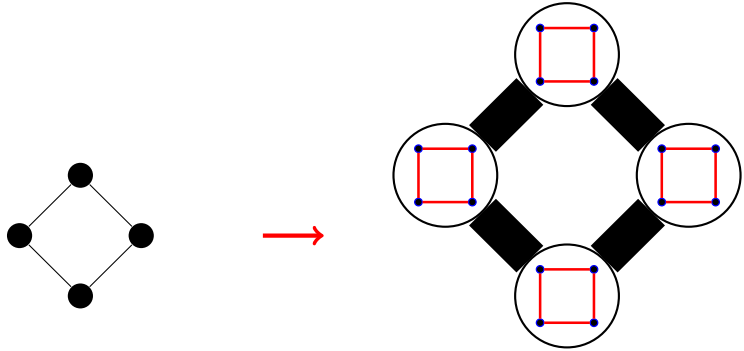

We say that blob $i$ and blob $j$ in $G_{N}$ are adjacent, denoted $i\circledast j$, if $\{v_{i},v_{j}\}\in E(G)$, for $i\in [0,n),j\in[0,n)$. The vertices of $G_{N}$ are $\lbrace 0, 1, \ldots, n|V(G_{N-1})|-1\rbrace$, where $|V(G)|=n$. The edges are $\lbrace \lbrace u+i\times n^{N} ,v + i \times n^{N} \rbrace : i\in[0,n), \lbrace u,v\rbrace \in E(G_{N-1}) \rbrace  \cup \lbrace \lbrace u + i \times n^{N}, v+i+j\times n^{N}\rbrace \pmod{|V(G_{N})|}\rbrace : i\in [0,n), j\in [0,n), u,v\in V(G), i\circledast j  \rbrace$. There exists copies of subgraphs that are isomorphic to $G$ within $G_{N-1}$ and also between each blob.

\section{Exact Counts}
In this section, we show the exact counts of $C_{4}$ for the nested blow-up graph of the graph $G=C_{4}$ and $G=\Theta_{2,2,2}$. Here, we denote $T_{N}$ as the number of induced subgraphs isomorphic to $C_{4}$ and $n_{N}$ as the number of vertices at level $N$ of the nested blow-up $G_{N}$ respectively (i.e.,\ $T_{0}=1$ and $n_{0}=4$), where $G_{0}=C_{4}$. We also denote the number of non-edges of $C_{4}$ as $m_{N}^{c}$, where $m_{0}^{c}=2$. Recall that a \emph{blob} is the set of $G_{N-1}$
vertices in $G_{N}$ from the blow-up of a vertex $v$ in $G=G_{0}$. We show the following lemma:

\begin{lemma}
The number of non-edges $m^{c}_{N}$ in the nested blow-up graph $G_{N}$ is
\begin{equation*}
       m_{N}^{c} = \binom{4^{N+1}}{2}-4^{N+1}\mathlarger{\sum}_{i=0}^{N}4^{i} = \dfrac{4^{N+1}(4^{N+1}-1)}{6}, N \in \mathbb{N}\cup \{0\}.
\end{equation*}
\label{lemma1}
\end{lemma}

\begin{proof}
We prove the equation from Lemma \ref{lemma1} by induction. \\
\textbf{Base case}: For $N=0$, the number of non-edges is  $m^{c}_{0}=\binom{4}{2}-4\mathlarger{\sum}^{0}_{i=0}4^{0}=6-4=\dfrac{4^{1}(4^{1}-1)}{6}=2$, which is precisely the number of non-edges in  $C_{4}$. \\
\noindent
Assume the induction hypothesis that for a particular $N$, the case $n = N$ holds, that is:
\begin{align*}
m^{c}_{N} &=\binom{4^{N+1}}{2}-4^{N+1}\mathlarger{\sum}^{N}_{i=0} 4^{i} 
= \dfrac{4^{N+1}(4^{N+1}-1)}{6}. \\
\intertext{The number of non-edges in $G_{N+1}$ is:}
m^{c}_{N+1} &=\binom{4^{N+2}}{2}-4\times |E(G_{N})|-4\times \text{edges between pairs of blobs}\\
\intertext{We obtain $|E(G_{N})|$ using $m_{N}^{c}$ such that $|E(G_{N})|=\binom{4^{N+1}}{2}-m_{N}^{c}$. Thus,}
 m_{N+1}^{c} &= \binom{4^{N+2}}{2} - 4\cdot 4^{N+1} \mathlarger{\sum}^{N}_{i=0} 4^{i} - 4\cdot(4^{N+1})^{2} \\
 &= \binom{4^{N+2}}{2} - 4^{N+2}\mathlarger{\sum}_{i=0}^{N} 4^{i} - 4^{N+2}4^{N+1} \\
 &= \binom{4^{N+2}}{2}-4^{N+2}\mathlarger{\sum}_{i=0}^{N+1} 4^{i}.
\end{align*}
Since both the base case and inductive step have been proven as true, thus by mathematical induction $m^{c}_{N}$ holds for all $N$.
\end{proof}

\noindent
Thus, we state the following theorem.

\begin{theorem}
The nested blow-up graph $G_{N}$, $N\geq 0$, of a $C_{4}$ has precisely $T_{N}= \frac{8\times 4^{N}}{5670}\times(1280\times 4^{3N}+672\times 4^{2N}+105\times 4^{N}-713)$ induced subgraphs isomorphic to $C_{4}$.
\end{theorem}

\begin{proof}
First, we will show that,
$$T_{N}=\begin{cases}4(T_{N-1})+(4^{N})^{4}+4\times m_{N-1}^{c}\times (4^{N})^{2}+4\times (m_{N-1}^{c})^{2}, &n>0\\
1 &n=0.\end{cases}$$ 
Since $G_{0}=C_{4}$, we have $T_{0}=1$.  We show that we can obtain $T_{N}$ from $T_{N-1}$ and prove each term from $T_{N}$ respectively. 

\begin{figure}[H]
\centering
\begin{tikzpicture}[scale=1.2,auto=left,every node/.style={circle,draw=black}]
    \node[scale=1.2,line width=1mm] (n0) at (0,0) {$B_{1}$};
    \node[scale=1.2,line width=1mm] (n1) at (1,-1) {$B_{4}$};
    \node[scale=1.2,line width=1mm] (n2) at (-1,-1) {$B_{2}$};
    \node[scale=1.2,line width=1mm] (n3) at (0,-2) {$B_{3}$};
    
    \draw[thick] (n0) -- (n1) -- (n3) -- (n2) -- (n0);
\end{tikzpicture}
\caption{Illustration of $G_{N}$ with four blobs labelled $B_{1},B_{2},B_{3}$ and $B_{4}$.}
\label{C4blobs}
\end{figure}
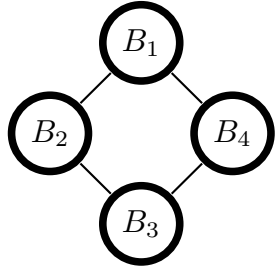

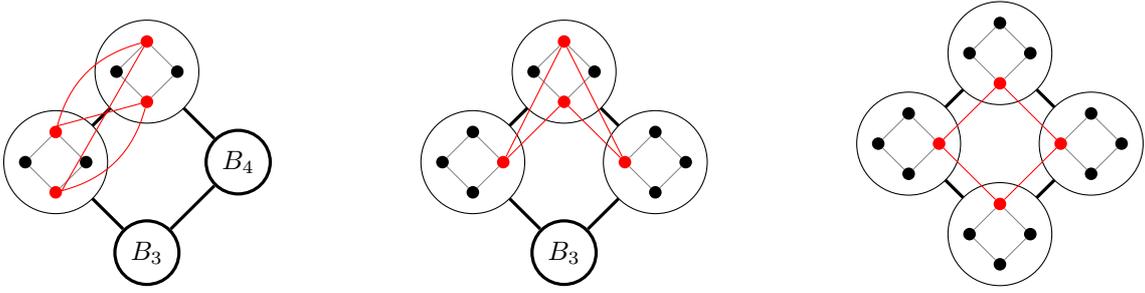
\begin{figure}[H]
\centering
  \begin{tikzpicture}[scale=0.4]
    \begin{scope}
      \node[g0,fill=red,name=1] at (1,0) {};
      \node[g0,name=2] at (0,1) {};
      \node[g0,fill=red,name=3] at (1,2) {};
      \node[g0,,name=4] at (2,1) {};
      \node[g1,fit=(1)(2)(3)(4)] (B1) {};
      \draw[very thin,gray] (1)--(2) -- (3) -- (4) -- (1);
    \end{scope}
    \begin{scope}[shift={(-3,-3)}]
      \node[g0,fill=red,name=5] at (1,0) {};
      \node[g0,name=6] at (0,1) {};
      \node[g0,fill=red,name=7] at (1,2) {};
      \node[g0,name=8] at (2,1) {};
      \node[g1,fit=(5)(6)(7)(8)] (B2) {};
      \draw[very thin,gray] (5)--(6) -- (7) -- (8) -- (5);
    \end{scope}
    \node[draw,very thick,name=B4] at (4,-2) {$B_4$};
    \node[draw,very thick,name=B3] at (1,-5) {$B_3$};
    \draw[very thick] (B1) -- (B2) -- (B3) -- (B4) -- (B1);
    \draw[red] (1) to[bend left] (5) -- (3) to[bend right] (7) -- (1);
  \end{tikzpicture}\hfill
  \begin{tikzpicture}[scale=0.4]
    \begin{scope}
      \node[g0,fill=red,name=1] at (1,0) {};
      \node[g0,name=2] at (0,1) {};
      \node[g0,fill=red,name=3] at (1,2) {};
      \node[g0,,name=4] at (2,1) {};
      \node[g1,fit=(1)(2)(3)(4)] (B1) {};
      \draw[very thin,gray] (1)--(2) -- (3) -- (4) -- (1);
    \end{scope}
    \begin{scope}[shift={(-3,-3)}]
      \node[g0,name=5] at (1,0) {};
      \node[g0,name=6] at (0,1) {};
      \node[g0,name=7] at (1,2) {};
      \node[g0,fill=red,name=8] at (2,1) {};
      \node[g1,fit=(5)(6)(7)(8)] (B2) {};
      \draw[very thin,gray] (5)--(6) -- (7) -- (8) -- (5);
    \end{scope}
    \begin{scope}[shift={(3,-3)}]
      \node[g0,name=9] at (1,0) {};
      \node[g0,fill=red,name=10] at (0,1) {};
      \node[g0,name=11] at (1,2) {};
      \node[g0,,name=12] at (2,1) {};
      \node[g1,fit=(9)(10)(11)(12)] (B4) {};
      \draw[very thin,gray] (9)--(10) -- (11) -- (12) -- (9);
    \end{scope}
    \node[draw,very thick,name=B3] at (1,-5) {$B_3$};
    \draw[very thick] (B1) -- (B2) -- (B3) -- (B4) -- (B1);
    \draw[red] (1) to (8) -- (3) to (10) -- (1);
  \end{tikzpicture}\hfill
  \begin{tikzpicture}[scale=0.4]
    \begin{scope}
      \node[g0,fill=red,name=1] at (1,0) {};
      \node[g0,name=2] at (0,1) {};
      \node[g0,name=3] at (1,2) {};
      \node[g0,name=4] at (2,1) {};
      \node[g1,fit=(1)(2)(3)(4)] (B1) {};
      \draw[very thin,gray] (1)--(2) -- (3) -- (4) -- (1);
    \end{scope}
    \begin{scope}[shift={(-3,-3)}]
      \node[g0,name=5] at (1,0) {};
      \node[g0,name=6] at (0,1) {};
      \node[g0,name=7] at (1,2) {};
      \node[g0,fill=red,name=8] at (2,1) {};
      \node[g1,fit=(5)(6)(7)(8)] (B2) {};
      \draw[very thin,gray] (5)--(6) -- (7) -- (8) -- (5);
    \end{scope}
    \begin{scope}[shift={(3,-3)}]
      \node[g0,name=9] at (1,0) {};
      \node[g0,fill=red,name=10] at (0,1) {};
      \node[g0,name=11] at (1,2) {};
      \node[g0,,name=12] at (2,1) {};
      \node[g1,fit=(9)(10)(11)(12)] (B4) {};
      \draw[very thin,gray] (9)--(10) -- (11) -- (12) -- (9);
    \end{scope}
    \begin{scope}[shift={(0,-6)}]
      \node[g0,name=13] at (1,0) {};
      \node[g0,name=14] at (0,1) {};
      \node[g0,fill=red,name=15] at (1,2) {};
      \node[g0,name=16] at (2,1) {};
      \node[g1,fit=(13)(14)(15)(16)] (B3) {};
      \draw[very thin,gray] (13)--(14) -- (15) -- (16) -- (13);
    \end{scope}
    \draw[very thick] (B1) -- (B2) -- (B3) -- (B4) -- (B1);
    \draw[red] (1) to (8) -- (15) to (10) -- (1);
  \end{tikzpicture}

\caption{[Left to right] Illustration of ways to obtain $C_{4}$ from either 2, 3, or all 4 blobs respectively. The selected vertices are shown in red.}
\label{C4ways}
\end{figure}

We know that $G_{N}$ has four blobs (see Figure \ref{C4blobs}) isomorphic to $G_{N-1}$.  Each of these blobs has $T_{N-1}$ copies of $C_{4}$, which contributes to $4T_{N-1}$ induced copies of $C_{4}$. Thus giving the first term $4T_{N-1}$. 

For simplicity of the remainder of the proof, we use Figure \ref{C4ways} to show the multiple ways of obtaining an induced $C_{4}$ from either 2, 3 or all 4 blobs. Example of how vertices can be chosen are shown in red. 

We select a vertex from each of the 4 blobs that are isomorphic to $G_{N-1}$ in $G_{N}$ which results in an induced $C_{4}$. There are $4^{N}$ vertices in each $G_{N-1}$ which gives $(4^{N})^{4}$ choices. This results in the second term $(4^{N})^{4}$. 

For the third term $4\times m^{c}_{N-1}\times (4^{N})^{2}$, we choose 1 blob that is isomorphic to $G_{N-1}$. We choose a non-edge from this blob. Then, we select two adjacent blobs and from each blob we select a vertex. There are $(4^{N})$ vertices in a blob. As we choose a different vertex   from each of the 2 blobs, there are $(4^N)^2$ choices of vertices, as well as $4\times m^{c}_{n-1}$ non-edges, that can be chosen from the third blob. 

Lastly, we choose 2 adjacent blobs and a non-edge from each. There are 4 choices for adjacent blobs in the blow-up of $C_{4}$. Thus, resulting in $4\times (m^{c}_{N-1})^2$ choices.

We now expand and simplify $T_{N}$ to find the recurrence relation. 

\noindent
Expanding the first few terms, we obtain:
\begin{align*}
    T_{0}= {} & 1 \\
    T_{1}= {} & 4(T_{0})+4^{4}+4\times m_{0}^{c}\times 4^{2}+4\times (m_{0}^{c})^{2} 
         = {}  4 + 4^{4} +4^{3}\times m_{0}^{c}+4\times (m_{0}^{c})^{2} \\
         = {} & 4^{1} \times \sum_{i=0}^{1} 4^{3i} + \mathlarger{\sum}_{i=1}^{1} (4^{1+i+1}\times m_{i-1}^{c})+\mathlarger{\sum}_{i=1}^{1} (4^{i}\times (m_{1-i}^{c})^{2}) \\
    T_{2}= {} & 4(T_{1})+4^{8}+4\times m_{1}^{c}\times 4^{4} +4\times (m_{1}^{c})^{2} \\
   = {} &  4\left (4^{1} \times \sum_{i=0}^{1} 4^{3i} + \mathlarger{\sum}_{i=1}^{1} (4^{1+i+1}\times m_{i-1}^{c})+\mathlarger{\sum}_{i=1}^{1} (4^{i}\times (m_{1-i}^{c})^{2})\right) +4^{8}+4\times m_{1}^{c}\times 4^{4} +4\times(m_{1}^{c})^{2}\\
     = {} &  \left (4^{2} \times \sum_{i=0}^{1} 4^{3i} +4^{8}\right)+ \left(4\mathlarger{\sum}_{i=1}^{1} (4^{1+i+1}\times m_{i-1}^{c})
     +4\times m_{1}^{c}\times 4^{4}\right)
     +\left(4\mathlarger{\sum}_{i=1}^{1} (4^{i}\times (m_{1-i}^{c})^{2})  +4\times(m_{1}^{c})^{2}\right )\\
        = {} & 4^{2} \times \sum_{i=0}^{2} 4^{3i} + \mathlarger{\sum}_{i=1}^{2} (4^{2+i+1}\times m_{i-1}^{c})+\mathlarger{\sum}_{i=1}^{2} (4^{i}\times (m_{2-i}^{c})^{2}) \\
    T_{3}= {} & 4(T_{2})+4^{12}+4\times m_{2}^{c}\times 4^{6}+4\times (m_{2}^{c})^{2} \\
        = {} & 4\left(4^{2} \times \sum_{i=0}^{2} 4^{3i} + \mathlarger{\sum}_{i=1}^{2} (4^{2+i+1}\times m_{i-1}^{c})+\mathlarger{\sum}_{i=1}^{2} (4^{i}\times (m_{2-i}^{c})^{2})\right)+4^{12}+4\times m_{2}^{c}\times 4^{6}+4\times (m_{2}^{c})^{2} \\
        = {} & \left(4^{3} \times \sum_{i=0}^{2} 4^{3i}+4^{12}\right) + \left(4\mathlarger{\sum}_{i=1}^{2} (4^{2+i+1}\times m_{i-1}^{c})+4\times m_{2}^{c}\times 4^{6}\right)+\left(4\mathlarger{\sum}_{i=1}^{2} (4^{i}\times (m_{2-i}^{c})^{2})+4\times (m_{2}^{c})^{2}\right) \\
        = {} & 4^{3} \times \sum_{i=0}^{3} 4^{3i} + \mathlarger{\sum}_{i=1}^{3} (4^{3+i+1}\times m_{i-1}^{c})+\mathlarger{\sum}_{i=1}^{3} (4^{i}\times (m_{3-i}^{c})^{2}) \\
        \vdots \\
      T_{N} = {} & \underbrace{4^{N} \times \sum_{i=0}^{N} 4^{3i}}_{Q_{N}} + \underbrace{\mathlarger{\sum}_{i=1}^{N} (4^{N+i+1}\times m_{i-1}^{c})}_{R_{N}}+\underbrace{\mathlarger{\sum}_{i=1}^{N} (4^{i}\times (m_{N-i}^{c})^{2})}_{S_{N}}
\end{align*}
\noindent
We simplify for each $Q_{N},R_{N}$ and $S_{N}$. \\

\noindent
Simplifying using geometric sum,
\begin{align}
    Q_{N} 
          &= 4^{N}\times \mathlarger{\sum}_{i=0}^{N} 4^{3i} 
   = 4^{N}\times \dfrac{(4^{3(N+1)}-1)}{4^{3}-1} 
         = \dfrac{4^{N}}{63}(64\times4^{3N}-1). 
\intertext{Using Lemma \ref{lemma1},}
    R_{N} 
        &= \mathlarger{\sum}_{i=1}^{N} (4^{N+i+1}\times m_{i-1}^{c}) \nonumber \\
   &= \mathlarger{\sum}_{i=1}^{N} \left(\frac{4^{N+i+1}}{6}\times\dfrac{4^{i}(4^{i}-1)}{6}\right) 
       = 4^{N+1}\times \mathlarger{\sum}_{i=1}^{N} \dfrac{4^{3i}-4^{2i}}{6} 
        = \dfrac{4^{N+1}}{6}\times\left(\mathlarger{\sum}_{i=1}^{N} 4^{3i}-\mathlarger{\sum}_{i=1}^{N} 4^{2i}\right). \nonumber\\
\intertext{Again, simplify using geometric sum,}
    R_{N} &= \dfrac{4^{N+1}}{6}\left(\dfrac{4^{3}(4^{3N}-1)}{4^{3}-1}-\dfrac{4^{2}(4^{2N}-1)}{4^{2}-1}\right) \nonumber \\
        &
      = \dfrac{4^{N+1}}{1890}(320\times(4^{3N}-1)-336\times(4^{2N}-1)) \nonumber \\
        &= \dfrac{4^{N+1}}{1890}(320\times 4^{3N}-336\times 4^{2N}+16).
\intertext{Lastly,}
    S_{N} 
    &= \mathlarger{\sum}_{i=1}^{N} (4^{i}\times (m_{N-i}^{c})^{2}) \nonumber \\
\intertext{By Lemma \ref{lemma1},} 
    S_{N} &= \mathlarger{\sum}_{i=1}^{N} \left(4^{i}\times \left(\frac{4^{N-i+1}(4^{N-i+1}-1)}{6}\right)^{2}\right) \nonumber \\
         &= \mathlarger{\sum}_{i=1}^{N} \left(4^{i}\times \frac{4^{2N-2i+2}\left(
     4^{N-i+1}-1\right)^{2}}{36}\right) \nonumber \\
     &= \dfrac{4^{N+1}}{9}\mathlarger{\sum}_{i=1}^{N}  4^{N-i}\left(4^{2N-2i+2}-2\times 4^{N-i+1}+1\right) \nonumber \\
    &= \dfrac{4^{N+1}}{9}\left( \mathlarger{\sum}_{i=1}^{N}  4^{3N-3i+2}-2\times \mathlarger{\sum}_{i=1}^{N}  4^{2N-2i+1}+\mathlarger{\sum}_{i=1}^{N}  4^{N-i}\right) \nonumber \\
     &= \dfrac{4^{N+1}}{9}\left(16\times \mathlarger{\sum}_{i=1}^{N}  4^{3(N-i)}-8\times\mathlarger{\sum}_{i=1}^{N}  4^{2(N-i)}+\mathlarger{\sum}_{i=1}^{N}  4^{N-i}\right). \nonumber \\
\intertext{Simplify using geometric sum,}
     S_{N} &= \dfrac{4^{N+1}}{9}\left(16\times   \frac{(4^{3N}-1)}{63}-8\times \frac{(4^{2N}-1)}{15}+\frac{(4^{N}-1)}{3}\right) \nonumber\\
    &= \frac{4^{N+1}}{2835}\left(80\times (4^{3N}-1)-168 \times (4^{2N}-1)+105\times (4^{N}-1)\right) \nonumber \\ 
     &= \frac{4^{N+1}}{2835}\left(80\times 4^{3N}-168\times 4^{2N}+ 105\times 4^{N} -17\right). 
\end{align}
\noindent
Thus,
\begin{align*}
 T_{N}= {} & \dfrac{4^{N}}{63}(64\times4^{3N}-1)+\dfrac{4^{N+1}}{1890}(320\times 4^{3N}-336\times 4^{2N}+16) \\
 & +\frac{4^{N+1}}{2835}\left(80\times4^{3N}-168\times4^{2N}+105\times 
     4^{N} -17\right) \\
     = {} & \dfrac{ 4^{N}}{5670}(10240\times 4^{3N}-5376\times 4^{2N}+840\times 4^{N}-34).
\end{align*}
\end{proof}

\subsection{Counting $C_{4}$ in the theta graph $\Theta_{2,2,2}$}
In this section, we construct the nested blow-up graph of the theta graph $\Theta_{2,2,2}$ by replacing each vertex $v_{i}$ in $\Theta_{2,2,2}$ with a $\Theta_{2,2,2}$. 

Recall that the $N$ level nested blow-up of a graph $G$ is denoted $G_{N}$. We compute the number of induced $C_{4}$ in $\Theta_{2,2,2}$. We denote $T_{N}$ as the number of induced $C_{4}$s and $n_{N}$ as the number of distinct vertices in $G_{N}$ respectively. Figure \ref{fig:Theta} gives the theta graph $\Theta_{2,2,2}$.

\begin{figure}[H]
\centering
\captionsetup{justification=centering} 
\begin{tikzpicture}
  [scale=.5,auto=left,every node/.style={circle,fill=black}]
   \node (n0) at (2,4) {};
   \node (n1) at (4,6) {};
   \node (n2) at (4,4) {};
   \node (n3) at (4,2) {};
   \node (n4) at (6,4) {};
   
   \foreach \from/\to in {n0/n1,n0/n2,n0/n3,n4/n1,n4/n2,n4/n3}
   \draw (\from) -- (\to);
\end{tikzpicture}
\caption{Illustration of theta graph $\Theta_{2,2,2}$.} 
\label{fig:Theta}
\end{figure}
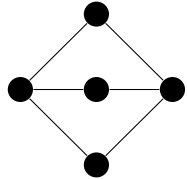

\begin{lemma}
Let $G=\Theta_{2,2,2}$. The number of non-edges $m^{c}_{N}$ in each level of the blow-up graph $G_{N}$ is given by
\begin{equation*}
       m_{N}^{c} = 4\cdot5^{N}\sum^{N}_{i=0}5^{i}=5^{N}(5^{N+1}-1).
\end{equation*}
\label{lemma2}
\end{lemma}

\begin{proof}
We prove the equation by induction. \\
\textbf{Base case}: When $N=0$, there are $m^{c}_{0}= 5^{0}(5^{1}-1) = 4$ non-edges which is precisely the number of non-edges in a $\Theta_{2,2,2}$. \\
\noindent
Assume the induction hypothesis that for a particular $N$, the single case $n = N$ holds, that is,
\begin{align*}
m_{N}^{c} &= 4\cdot5^{N}\sum^{N}_{i=0}5^{i}=5^{N}(5^{N+1}-1).\\
\intertext{The number of non-edges in $G_{N+1}$ is}
m_{N+1}^{c} &= \binom{5^{N+2}}{2} - |E(G_{N+1})| \\
\intertext{The number of edges in $G_{N}$ is calculated as follows: There are 5 vertices and 6 edges in $\Theta_{2,2,2}$. At the level $N$ blow-up, there are $6\cdot 5^{N}$ edges in each blob, and also $6\cdot 5^{N+1}$ edges between blobs, thus giving the term $|E(G_{N})|=6\cdot 5^{N}\mathlarger{\sum}_{i=0}^{N} 5^{i}$. Thus,}
m_{N+1}^{c}&= \binom{5^{N+2}}{2} - 6\cdot 5^{N+1} \mathlarger{\sum}_{i=0}^{N+1} 5^{i} \\
&= \dfrac{5^{N+2}(5^{N+2}-1)}{2} - \dfrac{(6\cdot 5^{N+1})(5^{N+2}-1)}{4} \\
&= 5^{N+1}(5^{N+2}-1).
\end{align*}
Since both the base case and inductive step has been proven as true, thus by mathematical induction $m^{c}_{N}$ holds for all $N$.
\end{proof}

\noindent
We state the following theorem. 

\begin{theorem}
The nested blow-up graph of a $\Theta_{2,2,2}$ has precisely $T_{N}= \frac{5^{N}}{1240}(6300\times 5^{3N}-2945\times5^{2N}+372\times5^{N}-3877)$ induced subgraphs isomorphic to $C_{4}$.
\end{theorem}

\begin{proof}
First, we will show that,
$$T_{N}=\begin{cases}5\times(T_{N-1})+3\times(5^{N})^{4}+6\times(m^{c}_{N-1})^{2}+9\times m^{c}_{N-1}\times(n_{N-1})^{2}, &n>0\\
3 &n=0.\end{cases}$$ Since $G_{0}=\Theta_{2,2,2}$, we have $T_{0}=3$.  We show that we can obtain $T_{N}$ from $T_{N-1}$ and prove each term from $T_{N}$ respectively. 

Note that at level $N$ blow-up, we replace each of the $5^{N+1}$ vertices in $G_{N-1}$ with $G_{0}$, this contributes to $5\times T_{N-1}$ induced copies of $C_{4}$ at $G_{N}$. 
Thus giving the first term $5\times T_{N-1}$.

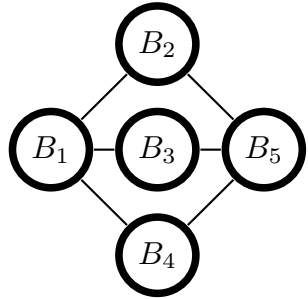
\begin{figure}[H]
\centering
\captionsetup{justification=centering} 
\begin{tikzpicture}
  [scale=.7,auto=left,every node/.style={circle,draw=black}]
   \node[scale=1.2,line width=1mm] (n0) at (2,4) {$B_{1}$};
   \node[scale=1.2,line width=1mm] (n1) at (4,6) {$B_{2}$};
   \node[scale=1.2,line width=1mm] (n2) at (4,4) {$B_{3}$};
   \node[scale=1.2,line width=1mm] (n3) at (4,2) {$B_{4}$};
   \node[scale=1.2,line width=1mm] (n4) at (6,4) {$B_{5}$};
   
   \foreach \from/\to in {n0/n1,n0/n2,n0/n3,n4/n1,n4/n2,n4/n3}
   \draw[thick] (\from) -- (\to);
  
\end{tikzpicture}
\caption{The $\Theta_{2,2,2}$ graph containing 5 blobs, with each blob labelled respectively.} 
\label{fig:Theta with blobs}
\end{figure}

We know that $G_{N}$ has five blobs isomorphic to $G_{N-1}$. For simplicity of the rest of the proof, we will refer to the blobs as labelled $B_{1},B_{2},B_{3},B_{4}$ and $B_{5}$, see Figure \ref{fig:Theta with blobs}.

We select 4 of these 5 blobs and one vertex from each blob.  There are five ways to select 4 blobs, but only 3 of these induce copies of $C_{4}$. Each of these three combinations of blobs contribute to 3 $C_{4}$ in $G_{N}$. There are $5^{N}$ vertices in each $G_{N-1}$ which gives $3\times(5^{N})^{4}$ choices. This results in the second term $3\times(5^{N})^{4}$. 

For the third term $6\times (m^{c}_{N-1})^{2}$, we choose 1 blob. We choose a non-edge from this blob and another non-edge from any adjacent blob. There are six pairs of blobs that are adjacent, thus resulting in the term $6\times (m^{c}_{N-1})^{2}$. 

Finally, we pick a non-edge in a blob $B$ and two vertices, one from each of two different blobs each adjacent to $B$.  There are nine ways that we can select these (refer to Figure \ref{fig:Theta with blobs}), namely, $\{B_{1},B_{2},B_{5}\}$, $\{B_{1},B_{3},B_{5}\}$, $\{B_{1},B_{4},B_{5}\}$, $\{B_{1},B_{2},B_{3}\}$, $\{B_{1},B_{2},B_{4}\}$, $\{B_{1},B_{3},B_{4}\}$, $\{B_{2},B_{3},B_{5}\}$, $\{B_{2},B_{4},B_{5}\}$, and $\{B_{3},B_{4},B_{5}\}$.

Each combination has $m^{c}_{N-1}$ choices of a non-edge from one blob and then a choice of a vertex from the $n_{N-1}$ vertices in each of the other two blobs. This results in the term $9\times m^{c}_{N-1}\times(n_{N-1})^{2}$.  
\noindent
We now expand and simplify $T_{N}$ to find the recurrence relation for $T_{N}$. Expanding the first few terms, we obtain:

\begin{align*}
    T_{0}={} & 3 \\
    T_{1}={} & 5(T_{0})+3\times 5^{4}+6\times(m_{0}^{c})^{2}+9\times m_{0}^{c}\times 5^{2} 
    = {}  3\cdot(5+5^{4})+6\times(m_{0}^{c})^{2}+9\times5^{2}\times m_{0}^{c} \\
    = {} & 3\cdot 5^{1}\times \mathlarger{\sum}_{i=0}^{1} 5^{3i} + 6\times \mathlarger{\sum}_{i=1}^{1}(5^{i-1}\times(m_{1-i}^{c})^{2}) + 9\times \mathlarger{\sum}_{i=1}^{1} (5^{1+i}\times m_{i-1}^{c})\\
    T_{2}={} & 5(T_{1})+3\times 5^{8}+6\times(m_{1}^{c})^{2}+9\times m_{1}^{c}\times 5^{4} \\
    ={} & 5\left(3\cdot 5^{1} \times \mathlarger{\sum}_{i=0}^{1}5^{3i} + 6\times \mathlarger{\sum}_{i=1}^{1}(5^{i-1}\times(m_{1-i}^{c})^{2}) + 9\times \mathlarger{\sum}_{i=1}^{1} (5^{1+i}\times m_{i-1}^{c})\right)+3\times 5^{8}+6\times(m_{1}^{c})^{2}+9\times m_{1}^{c}\times 5^{4} \\
    ={} & \left(3\cdot 5^{2}\mathlarger{\sum}_{i=0}^{1}5^{3i}+3\cdot5^{8}\right) +\left( 6\cdot 5\mathlarger{\sum}_{i=1}^{1}(5^{i-1}\times(m_{1-i}^{c})^{2})+6\cdot(m_{1}^{c})^{2}\right) \\
    {} & +\left( 9\cdot5 \mathlarger{\sum}_{i=1}^{1} (5^{1+i}\times m_{i-1}^{c})+9\cdot m_{1}^{c}\cdot 5^{4})\right) \\
    = {} & 3\cdot 5^{2} \times \mathlarger{\sum}_{i=0}^{2}5^{3i} + 6\times \mathlarger{\sum}_{i=1}^{2}(5^{i-1}\times(m_{2-i}^{c})^{2}) + 9\times \mathlarger{\sum}_{i=1}^{2} (5^{2+i}\times m_{i-1}^{c}) \\
    T_{3}={} & 5(T_{2})+3\times 5^{12}+6\times(m_{2}^{c})^{2}+9\times m_{2}^{c}\times 5^{6} \\
    ={} & 5\left(3\cdot 5^{2} \times \mathlarger{\sum}_{i=0}^{2}5^{3i} + 6\times \mathlarger{\sum}_{i=1}^{2}(5^{i-1}\times(m_{2-i}^{c})^{2}) + 9\times \mathlarger{\sum}_{i=1}^{2} (5^{2+i}\times m_{i-1}^{c})\right)+3\times 5^{12}+6\times(m_{2}^{c})^{2}+9\times m_{2}^{c}\times 5^{6} \\
     ={} & \left(3\cdot 5^{3}\mathlarger{\sum}_{i=0}^{2}5^{3i}+3\cdot 5^{12}\right) + \left(6\cdot 5 \mathlarger{\sum}_{i=1}^{2}(5^{i-1}\times(m_{2-i}^{c})^{2})+6\cdot(m_{2}^{c})^{2}\right)\\
     {} & +\left(9\cdot5 \mathlarger{\sum}_{i=1}^{2} (5^{2+i}\times m_{i-1}^{c})+9\cdot m_{2}^{c}\cdot 5^{6}\right) \\
    = {} & 3\cdot 5^{3} \times \mathlarger{\sum}_{i=0}^{3}5^{3i}+ 6\times \mathlarger{\sum}_{i=1}^{3}(5^{i-1}\times(m_{3-i}^{c})^{2}) + 9\times \mathlarger{\sum}_{i=1}^{3} (5^{3+i}\times m_{i-1}^{c})\\
     \vdots \\
     T_{N} = {} & \underbrace{3\cdot 5^{N} \times \mathlarger{\sum}_{i=0}^{N}5^{3i}}_{Q_{N}} + \underbrace{6\times \mathlarger{\sum}_{i=1}^{N}(5^{i-1}\times(m_{N-i}^{c})^{2})}_{R_{N}} + \underbrace{9\times \mathlarger{\sum}_{i=1}^{N} (5^{N+i}\times m_{i-1}^{c})}_{S_{N}}.
\end{align*}
\noindent
We simplify for each $Q_{N}, R_{N}$ and $S_{N}$.
\noindent
Simplifying using geometric sum,
\begin{align}
     Q_{N} &= 3\cdot 5^{N}\times \mathlarger{\sum}_{i=0}^{N} 5^{3i} 
         = 3\cdot 5^{N} \times \dfrac{5^{3}(5^{3N}-1)}{5^{3}-1}
       = \dfrac{3\cdot 5^{N}}{124}(125\times 5^{3N}-1). 
\intertext{Using Lemma \ref{lemma2},}
     R_{N} &= 6\times \mathlarger{\sum}_{i=1}^{N}(5^{i-1}\times(m_{N-i}^{c})^{2}) \nonumber \\
    &= 6 \times \mathlarger{\sum}_{i=1}^{N}(5^{i-1}\times (5^{N-i}(5^{N-i+1}-1))^{2}) \nonumber \\
     &= 6\cdot5^{N}\left(5\cdot\mathlarger{\sum}_{i=1}^{N} 5^{3(N-i)} -2\cdot\mathlarger{\sum}_{i=1}^{N} 5^{2(N-i)}+\frac{1}{5}\cdot \mathlarger{\sum}_{i=1}^{N} 5^{N-i}\right). \nonumber \\
\intertext{Again, simplify using geometric sum,}
     R_{N} &= 6\cdot 5^{N}\left(5\times\frac{5^{3N}-1}{5^{3}-1}-2\times\frac{5^{2N}-1}{5^{2}-1}+\frac{1}{5}\times\frac{5^{N}-1}{5-1}\right) \nonumber \\
     &= \dfrac{5^{N}}{620}(150\times 5^{3N}-310\times5^{2N}+186\times 5^{N}-26). \\
\intertext{Lastly,}
     S_{N} &= 9\times \mathlarger{\sum}_{i=1}^{N} (5^{N+i}\times m_{i-1}^{c}). \nonumber \\
\intertext{By Lemma \ref{lemma2},}
    S_{N} &= 9\times \mathlarger{\sum}_{i=1}^{N} (5^{N+i}\times (5^{i-1}(5^{i}-1))) 
    = 9\cdot 5^{N} \times \mathlarger{\sum}_{i=1}^{N} (5^{3i-1}-5^{2i-1})
   = \dfrac{9\cdot 5^{N}}{5} \times \left(\mathlarger{\sum}_{i=1}^{N} 5^{3i}- \mathlarger{\sum}_{i=1}^{N} 5^{2i}\right). \nonumber \\
    \intertext{Simplify using geometric sum,}
    S_{N} &= \dfrac{9\cdot 5^{N}}{5} \times \left(\dfrac{5^{3}(5^{3N}-1)}{5^{3}-1}- \dfrac{5^{2}(5^{2N}-1)}{5^{2}-1}\right)
   = \dfrac{3\cdot5^{N}}{1240}(750\times5^{3N}-775\times 5^{2N}+25).
\end{align}
\noindent
Thus,
\begin{align}
    T_{N}={} & \dfrac{3\cdot 5^{N}}{124}(125\times 5^{3N}-1) + \dfrac{5^{N}}{620}(150\times 5^{3N}-310\times5^{2N}+186\times 5^{N}-26) \nonumber \\
     & + \dfrac{3\cdot5^{N}}{1240}(750\times5^{3N}-775\times 5^{2N}+25) \nonumber \\
     ={} & \frac{5^{N}}{1240}(6300\times 5^{3N}-2945\times5^{2N}+372\times5^{N}-7).
\end{align}

\end{proof}

\section{Conclusion}
In this paper, we gave exact counts of $C_{4}$s in two different graph structures: (i) the nested blow-up graphs of $C_{4}$s and (ii) the theta graph $\theta_{2,2,2}$. Previously, only bounds were found for the blow-ups of graphs \cite{pippenger75}. We improved the bounds to give the exact counts of such graphs. In a general case, we can adapt a similar formula to construct the equations for blow-up graphs of higher order $k$s, to find the exact counts of cycles of higher order $k$.  Future direction of this work could include finding a generalised formula for any types of blow-up graphs, with \emph{some} cyclic property.

\bibliographystyle{plainnat}
\bibliography{main}

\end{document}